\documentclass[12pt]{amsart}
\usepackage{amssymb,latexsym}
\usepackage{enumerate}

\makeatletter
\newtheorem{thm}{Theorem}[section]
\newtheorem{prop}[thm]{Proposition}
\newtheorem{lem}[thm]{Lemma}
\newtheorem{cor}[thm]{Corollary}

\theoremstyle{definition}
\newtheorem{defin}[thm]{Definition}
\newtheorem{rem}[thm]{Remark}
\newtheorem{exa}[thm]{Example}

\newtheorem{quest}[thm]{Question}
\newtheorem*{xthm}{Theorem}

\DeclareMathOperator{\diam}{diam}

\newcommand{\N}{\mathbb{N}}
\newcommand{\K}{\mathcal{K}_{\sigma\delta}}
\newcommand{\ee}{\varepsilon}
\begin{document}

\baselineskip=17pt

\title[Strongly $\K$ spaces]{On a certain class of $\K$ Banach spaces}
\author{K.K. Kampoukos}
\address{University of Athens\\ Department of Mathematics\\Panepistemiopolis\\ 15784 Athens, Greece}
\email{kkamp@math.uoa.gr}

\author{S.K. Mercourakis}
\address{University of Athens\\ Department of Mathematics\\ Panepistemiopolos, 15784 Athens, Greece}
\email{smercour@math.uoa.gr}
\begin{abstract}
Using a strengthening of the concept of $\K$ set, introduced in this paper, 
we study a certain subclass of the class of $\K$ Banach spaces; the so called strongly 
$\K$ Banach spaces. This class of spaces includes subspaces of strongly 
weakly compactly generated (SWCG) as well as Polish Banach  spaces and it is related 
to strongly weakly $\mathcal{K}$--analytic (SWKA) Banach spaces as the known 
classes of $\K$ and weakly $\mathcal{K}$--analytic (WKA) Banach spaces are related.
\end{abstract}

\subjclass[2010]{ Primary 46B20, 54H05, 03E15; secondary 46B26}

\keywords{ $\K$, strongly $\K$ spaces, SWCG, SWKA Banach spaces}

\maketitle
\section*{Introduction}  
In this paper we study a certain subclass of $\K$ Banach spaces 
(\cite{T1}, \cite{AAM1} and \cite{AAM2}) and simultaneously a 
subclass of strongly weakly $\mathcal{K}$--analytic 
(SWKA) Banach spaces (\cite{MS}). Strongly $\mathcal{K}$--analytic 
topological spaces (a class of spaces contained in $\mathcal{K}$--analytic spaces) 
were introduced in \cite{MS} in order to study subspaces of strongly 
weakly compactly generated (SWCG) Banach spaces; the latter class was 
introduced by Schl\"uchtermann and Wheller in \cite{SW} (see also \cite{FMZ2} and \cite{LR}). 
SWKA Banach spaces are those Banach spaces which are strongly $\mathcal{K}$--analytic in 
their weak topology and include subspaces of SWCG Banach spaces (\cite{MS}). 
We recall that weakly $\mathcal{K}$--analytic (WKA) spaces, defined as those Banach 
spaces which are $\mathcal{K}$--analytic in their weak topology (\cite{T1}, \cite{V}), 
include $\K$ Banach spaces (i.e. those Banach spaces $X$ which are $\K$ sets 
in $(X^{**}, w^*)$)). The interrelation and analogies between the above mentioned 
``strong'' classes of spaces and also the structure of SWKA Banach spaces 
are discussed and investigated in \cite{MS} (see also \cite{KM}).

In Section \ref{S:A} of the present paper we strengthen the $\K$ property 
of a subset of a topological space by introducing the concept of a strongly 
$\K$ set (Def. \ref{D:SKcd}). Our motivation was to find a characterization 
of strongly $\mathcal{K}$--analytic topological spaces in the sense of Choquet; 
recall that a topological space is $\mathcal{K}$--analytic according to Choquet, 
if it is a continuous image of a $\K$ subset of some compact space (see \cite{JR}). 
So we prove that a topological space is strongly $\mathcal{K}$--analytic if 
and only if, it is an image under a (continuous) compact covering map of a 
strongly $\K$ subset of some compact space (Th. \ref{T:Choquet}). 
We also prove that any Polish subspace of any compact space $K$ is a 
strongly $\K$ subset of $K$ (Th. \ref{T:compactification}). Then we characterize 
strongly $\K$ sets in a topological space $X$ as those sets $A\subseteq X$ 
such that $\mathcal{K}(A)$ (the set of compact non empty subsets of $A$) is 
a $\K$ subset of $\mathcal{K}(X)$, where the latter space is endowed with 
Vietoris topology (Th. \ref{T:Vietoris}).

In Section \ref{S:B}, we deal with a natural generalization of SWCG Banach 
spaces. A SWCG Banach space $X$ is one which is strongly generated by a weakly 
compact set $K \subseteq X$. We consider a superspace $Z$ of $X$ and we let 
the strongly generating weakly compact set $K$ be contained in $Z$. 
In that case we say that $X$ is SWCG with respect to $Z$. It is clear that 
the new class is closed under (closed) subspaces; on the other hand by 
\cite[Th. 3.9]{MS} the class of SWCG is not. So the new class strictly contains 
the class of SWCG. Besides this fact, we show that the new class has similar 
properties with the old one (cf. \cite{SW}). So we prove a characterization 
of the new class using the Mackey topology on the dual space (Th. \ref{T:Mackey}) 
and we show that their members are weakly sequentially complete spaces 
(Th. \ref{T:wsc}). We also prove by an example that the new class is not closed 
for countable $\ell_p$--sums, $1<p<+\infty$, even in the case when each space 
is SWCG (Th. \ref{T:sum}). Our example, a separable weakly sequentially complete 
space not isomorphic to a (subspace of) SWCG, has stronger properties than 
\cite[Example 2.6]{SW}.

In Section \ref{S:C}, we introduce the common subclass of $\K$ and SWKA Banach 
spaces, which we mentioned in the first lines of this introduction. A Banach 
space $X$ will be called a strongly $\K$ space, if it is a strongly $\K$ 
subset of $(X^{**}, w^*)$. It should be clear that such a Banach space is 
$\K$ and SWKA. Note that every separable Banach space is $\K$, but not 
necessarily a strong $\K$ space; indeed, the space $c_0(\N)$ is not even 
SWKA (\cite{MS}). The class of strongly $\K$ Banach spaces is related to 
the class of SWKA Banach spaces as the familiar classes of $\K$ and WKA 
Banach spaces are related. Every Banach space which is SWCG with respect 
to a superspace and every Polish Banach space is strongly $\K$ 
(Ths. \ref{T:*Kcd} and \ref{T:Polish}). Moreover the class of strongly 
$\K$ Banach spaces is closed under countable $\ell_p$--sums, 
$1\leq p <+\infty$ (Th. \ref{T:prodKcd}). Finally, we investigate 
locally (in the weak topology of a Banach space) the property of 
strong $\mathcal{K}$--analyticity and the property of a set to 
be strongly $\K$. We conclude with some open questions. 

\section*{Preliminaries and notation} We denote by $\Sigma$ the set $\mathbb{N}^\mathbb{N}$ of 
infinite sequences of positive integers, endowed with the cartesian topology, which makes 
$\Sigma$ (usually called the "Baire space") a Polish space (i.e., homeomorphic to a complete 
separable metric space). $S$ stands for the set $\bigcup_{n=0}^\infty \mathbb{N}^n$, 
($\mathbb{N}^0=\emptyset)$ of finite sequences of positive integers. We give $S$ the partial
order of ''initial segments'' which makes $S$ into a tree: for $s=(s_1,\ldots, s_n)$, $t=(t_1, \ldots, t_m)$ 
members of $S$ we define $s\leq t$ if $n\leq m$ and 
$s_i=t_i$ for all $i=1,2\ldots, n$. If $s=(s_1,\ldots, s_k)\in S$, $\sigma=(n_1, n_2, \dots, n_k, \ldots) \in \Sigma$ and $m\in \mathbb{N}$, then we write: (i) $s<\sigma$ if $s_i=n_i$ for all $i=1,2,\ldots, k$ (ii) $\sigma|m$ 
for the finite sequence $(n_1, n_2, \ldots, n_m)$. For $\tau$, $\sigma\in \Sigma$ we set $\sigma \leq \tau$ 
if $\sigma (n) \leq \tau (n)$ for all $n \in \N$.

In the present paper by the term topological space we mean a Hausdorff and completely regular space.
For a  topological space $X,\: \mathcal{K}(X)$ is the set of compact 
non empty subsets of $X$. If $Y$ is a topological space, a map $F:Y \mapsto \mathcal{K}(X)$
is said to be upper semicontinuous (usco) if for every $y\in Y$ and $V$ open subset of $X$ with $F(y)\subseteq V$ 
there exists a neighbourhood $W$ of $y$ such that $F(W)=\bigcup\{F(t): t\in W\} \subseteq V$.

A topological space $X$ is called $\mathcal{K}$--analytic if there exists an usco map 
$F \colon \Sigma \to \mathcal{K}(X)$ such that $F(\Sigma)=X$. Moreover, if for each compact subset 
$L$ of $X$ there exists $\sigma \in \Sigma$ such that $L \subseteq F(\sigma)$ the space $X$ is called 
strongly $\mathcal{K}$--analytic (\cite{MS}).

A topological space $X$ is called \v{C}ech complete if it is a G$_\delta$ subset in some 
(every) compactification $cX$ of $X$.

Let $X$ be a Banach space, $\varepsilon \geq 0$. A bounded subset $M$ of $X$ is called  
$\varepsilon$--weakly relatively compact if 
$\overline{M}^{w^*} \subseteq X+ \varepsilon B_{X^{**}}\,$ .

We use \cite{LT} and \cite{FHHMZ} as basic references for the theory, notation and terminology of
Banach spaces and \cite{JR} for the theory of $\mathcal{K}$--analytic and countably determined 
topological spaces. If $X$ is a (real) Banach space, then $B_X$ and $B_{X^*}$ are the closed
unit balls of $X$ and its dual $X^*$ respectively.
\section{Strongly $\K$ topological spaces}
\label{S:A}
We recall that a subset $A$ of a topological space $X$ is called $\K$ if there exists a countable family
$\{K_{n,m} : n,m \in \N \}$ of compact subsets of $X$ such that $A=\bigcap_{n=1}^\infty \bigcup_{m=1}^\infty K_{n,m}$. In other words $A$ is a $\K$ subset of $X$ if $A$ can be written as a countable intersection of $\sigma$ -compact subsets of $X$.

In the next definition we strengthen the concept of the $\K$ subset of a topological space.
\begin{defin} \label{D:SKcd}
A subset $A$ of a topological space $X$ is called \emph{strongly $\K$} if there exists a countable family
$\{K_{n,m} : n,m \in \N \}$ of compact subsets of $X$ such that: 
\begin{enumerate}[\upshape (i)]
	\item $A=\bigcap_{n=1}^\infty \bigcup_{m=1}^\infty K_{n,m}$.
	\item For each compact subset $K$ of $A$ and for each $n \in \N$ there exists $m \in \N$ such that
	$K\subseteq K_{n,m}$.
\end{enumerate}
\end{defin}
\begin{rem} \label{R:Kcd}
\textbf{(i)} It is clear that a strongly $\K$ subset of a topological space $X$ is a $\K$ subset of $X$.

\textbf{(ii)} It is easy to check that if $A$ is a $\K$ (resp. strongly $\K$) subset of a topological space $X$ and $B$ is a relatively closed subset of $A$, then $B$ is a $\K$ (resp. strongly $\K$) subset of $X$.
\end{rem}
We now give some simple examples of strongly $\K$ subsets. We first recall that a topological space 
$X$ is said to be hemicompact if it can be written as $X=\bigcup_{n=1}^\infty K_n$, where $K_n$ are 
compact subsets of $X$ such that each compact set in $X$ is contained in some $K_n$. 
It is obvious that every hemicompact subspace $A$ of a topological space $X$ is a strongly 
$\K$ subset of $X$.
\begin{exa} \label{hemicompact}
\textbf{(a)} Every $\sigma$--compact and locally compact space is hemicompact (\cite[p. 250]{E}). A considerable class  of such spaces is the class of locally compact, metrizable and separable topological spaces. In particular, an open and $F_\sigma$ subset of a compact space is in the relative topology $\sigma$--compact and locally compact space.
\textbf{(b)} Every countable Hausdorff space whose compact sets are finite is hemicompact.
\textbf{(c)} Every dual Banach space $X^*$, endowed with the $w^*$ topology is hemicompact. 
\end{exa}

\begin{prop}\label{P:Kanalytic}
Every strongly $\K$ subset $A$ of a topological space $X$ is in the relative topology strongly $\mathcal{K}$--analytic.
\end{prop}
\begin{proof}
Let $(K_{n, m})$ be a double sequence of compact subsets of $X$, which satisfies conditions (i), (ii) of Def.\ref{D:SKcd}. 
We can assume and we do that for each $n \in \N$ the sequence $(K_{n, m})_m$ is increasing. For each 
$s=(m_1, m_2,\ldots m_n) \in S $, put $B_s=\bigcap_{i=1}^n K_{i,m_i}$. Let $K$ be a compact subset of $A$. Then according to condition (ii) of Def.\ref{D:SKcd} for each $i \in \N$ there exists $m_i \in \N$ such that $K\subseteq K_{i, m_i}$. 
If we set $\sigma =(m_1, m_2 ,\ldots m_i, \dots )$, then we have 
	\[K\subseteq \bigcap_{i=1}^\infty K_{i, m_i}, \quad \text{hence} \quad 
	K\subseteq \bigcap _{n=1}^\infty B_{\sigma|n} \subseteq A.
\]
It is not difficult to see that the map 
	\[F\colon \varSigma \to \mathcal{K}(A) \quad \text{with}\quad F(\sigma)=\bigcap_{n=1}^\infty B_{\sigma|n}
\]
is usco with $F(\Sigma)=A$, hence the space $A$ is strongly $\mathcal{K}$--analytic.
\end{proof}

\begin{cor}
Let $A$ be a metrizable and strongly $\K$ subset of a topological space $X$. Then $A$ is a Polish space.
\end{cor}
\begin{proof}
It is immediate from Prop.\ref{P:Kanalytic} that $A$ is a strongly 
$\mathcal{K}$--analytic and metrizable. Then the result follows from \cite[Prop. 1.11.1]{MS}.
\end{proof}

We proceed to prove a characterization of strongly $\mathcal{K}$--analytic spaces  analogous to the Choquet 
characterization of $\mathcal{K}$--analytic spaces. We need some results that seems to have independent interest.
\begin{prop} \label{P:intersection}
Let $(A_i)_i$ be a sequence of (strongly) $\K$ subsets of a topological space $X$. 
Then $A=\bigcap_{i=1}^\infty A_i$ is 
a (strongly) $\K$ subset of $X$.
\end{prop}
\begin{proof}
Given $i \in \N$ there exists a double sequence $(K_{n, m}^i)$ of compact subsets of $X$ with  
$A_i=\bigcap_{n=1}^\infty \bigcup_{m=1}^\infty K_{n, m}^i$ 
(such that for each pair $i$, $n\in \N$ and for each compact subset $L$ of $A_i$ there exists 
$m \in \N$ such that $L \subseteq K_{n, m}^i$). 
Then $\bigcap_{i=1}^\infty A_i=\bigcap_{i,n}\bigcup_{m=1}^\infty K_{n, m}^i$ 
and the conclusion follows.
\end{proof}

\begin{prop}\label{P:hemi}
Let $X$ be a topological space, $A$, $Y$ subsets of $X$ such that $A\subseteq Y$ and $A$ 
is a (strongly) $\K$ subset of $X$. 
If $Y$ is a (hemicompact subset of $X$) $\sigma$--compact subset of $X$, 
then $A$ is a (strongly) $\K$ subset of $Y$.
\end{prop}
\begin{proof}
It is enough to notice the following.
Let $K_{n, m}$ and $\Omega_{l}$, $n$, $m$, $l\in \N$ be compact subsets of $X$ such that 
$ A=\bigcap_{n=1}^\infty \bigcup _{m=1}^\infty K_{n, m}$  
and $Y=\bigcup_{l=1}^\infty \Omega_l$.
Then $A=A\bigcap Y =\bigcap_{n}\bigcup_{m, l}\left( K_{n, m}\bigcap \Omega_l \right)$
and the conclusion follows.
\end{proof}
In the sequel we prove that the product of two (strongly) $\K$ subsets is also 
(strongly) $\K$, hence inductively we have that the class of (strongly) 
$\K$ subsets is closed under finite products. 
Moreover, if we consider a sequence $(A_i)$ of (strongly) $\K$ subsets of compact spaces 
$(X_i)$ respectively, then the product $\prod_{i=1}^\infty A_i$ is a
(strongly) $\K$ subset of the space $\prod_{i=1}^\infty X_i$.
\begin{prop}\label{finiteprod}
Let $A_1$, $A_2$ be (strongly) $\K$ subsets of the topological spaces $X_1$ and $X_2$ respectively. Then the set 
$A_1\times A_2$ is a (strongly) $\K$ subset of the topological space $X_1 \times X_2$.
\end{prop}
\begin{proof}
It is enough to prove only the case when $A_1$, $A_2$ are $\K$ sets.
For each $i=1,2$ there exists a double sequence $(K_{n, m}^i)_{n, m}$ of compact subsets of $X_i$ which satisfies 
the requirments for a $\K$ set. 
For each $n\in \N$ we consider the countable family 
$\{K_{n, m_1}^1 \times K_{n, m_2}^2 : m_1 , m_2 \in \N \}$ of compact subsets of $X_1 \times X_2$. 
Then $A_1 \times A_2= \bigcap_{n=1}^\infty \bigcup \{K_{n, m_1}^1 \times K_{n, m_2}^2 : m_1, m_2 \in \N \}$
and the conclusion follows immediately.
\end{proof}
\begin{prop}
Let $A_i$ be a (strongly) $\K$ subset of a compact space $X_i$ for each $i \in \N$. Then the set 
$\prod_{i=1}^\infty A_i$ is a (strongly) $\K$ subset of the space $\prod_{i=1}^\infty X_i$.
\end{prop}
\begin{proof}
Let us prove this time the strongly $\K$ case.
For each $n\in\mathbb{N}$ by Prop. \ref{finiteprod} the set $\prod_{i=1}^{n}A_i$ is a strongly $\K$ subset 
of the space $\prod_{i=1}^{n}X_i$. The space $\prod_{i=n+1}^{\infty}X_i$ 
is compact, so the set $\prod_{i=1}^{n}A_i \times \prod_{i=n+1}^{\infty}X_i$ 
is a strongly $\K$ subset of the topological space $\prod_{i=1}^{n}X_i \times \prod_{i=n+1}^{\infty}X_i$, 
that is, of the space $\prod_{i=1}^{\infty}X_i$. By Proposition \ref{P:intersection}, 
a countable intersection of strongly $\K$ subsets of a topological space is a strongly $\K$ 
subset, so the set 
\[
\bigcap_{n=1}^\infty\left( \prod_{i=1}^{n}A_i \times \prod_{i=n+1}^{\infty}X_i\right) 
\]
is a strongly $\K$ subset of $\prod_{i=1}^{\infty}X_i$. Furthermore 
\[ \bigcap_{n=1}^\infty\left( \prod_{i=1}^{n}A_i \times \prod_{i=n+1}^{\infty}X_i\right)=\prod_{i=1}^{\infty}A_i\, ,
\]
and the conclusion follows.
\end{proof}
\begin{rem}
The above result fails to be true if compactness of $X_i$ is omitted from the hypotheses. Indeed, let 
$A_i =\mathbb{N}$ and $X_i =\mathbb{R}$ for each $i \in \mathbb{N}$. Then 
$\prod_{i=1}^\infty A_i=\mathbb{N}^\mathbb{N}=\varSigma \subseteq \prod_{i=1}^\infty X_i=\mathbb{R}^\mathbb{N}$.
But the Baire space $\varSigma$  is not a $\K$ subset of $\mathbb{R}^\mathbb{N}$. If it was, then it would be $\sigma$--compact, which contradicts Baire's theorem. We note that a $\K$ subset of a topological space $X$ is contained in some 
$\sigma$--compact subset of $X$ and the space $\mathbb{N}^\mathbb{N}$ is a closed subspace of $\mathbb{R}^\mathbb{N}$.
\end{rem}
\begin{thm} \label{T:compactification} 
Let $M$ be a Polish subspace of a compact space $K$. Then $M$ 
is a strongly $\K$ subset of $K$.
\end{thm}
\begin{proof}
It is clear that we can assume that $M$ is dense in $K$. Consider a complete metric $d$, that induces the topology of $M$. 
For each $n \in \mathbb{N}$ and for each $x\in M$ choose an open subset $V_n(x)$ of $K$ such that 
$x \in V_n(x)$ and $\diam (V_n(x)\cap M)\leq \frac{1}{n}$. Define $V_n=\bigcup_{x\in M}V_n(x)$. 
It is rather easy to see from completeness 
of $(M, d)$ that $M=\bigcap_{n=1}^\infty V_n$ (see \cite[Lemma 2, p. 40]{Chr}). 
For each $x\in M$ consider an open subset $W_n(x)$ of  $K$ such that 
	\[x\in W_n(x)\subseteq \overline{W_n(x)}\subseteq V_n(x) \, .
\]
The space $M$ is a seperable metric space, hence Lindel\"of. Consequently the open covering 
$\{\, V_n(x): x\in M\}$ of $M$ has a countable subcovering. Then for each 
$n\in\mathbb{N}$ choose a sequence $(x_{n,m})_m$ of $M$ such that 
$M\subseteq\bigcup_{m=1}^\infty W_n(x_{n, m}) $.
Then
	$M\subseteq \bigcap_{n=1}^\infty\bigcup_{m=1}^\infty\overline{W_n(x_{n,m})}\subseteq\bigcap_{n=1}^\infty\bigcup_{m=1}^\infty V_n(x_{n,m})\subseteq M$, hence 
$M= \bigcap_{n=1}^\infty\bigcup_{m=1}^\infty\overline{W_n(x_{n,m})}$.
For each $n, m\in\mathbb{N}$ put 
$G_{n, m}=\bigcup_{k=1}^m\overline{W_n(x_{n,k})}$ and then
$M= \bigcap_{n=1}^\infty\bigcup_{m=1}^\infty G_{n, m}$.
Let $L$ be a compact subset of $M$ and $n\in\mathbb{N}$. Then the open covering 
$\left(W_n(x_{n, m})\right)_{m=1}^\infty$ of $L$ has a finite subcovering, hence there exists 
$m\in\mathbb{N}$ such that 
$L\subseteq\bigcup_{k=1}^m\overline{W_n(x_{n,k})}$, so $L\subseteq G_{n, m}$ 
and the proof is complete.
\end{proof}
\begin{rem}\label{R:Cech}
A similar argument proves that every \v{C}ech complete and Lindel\"of space $X$ 
is strongly $\K$ in every compact superspace of $X$ (cf. the proof of \cite[Th. 3.9.2]{E}).
\end{rem}
We recall that a continuous map $f \colon X \to Y$ is called compact covering if for every compact subset $L$ 
of $Y$ there exists a compact subset $K$ of $X$ such that $f(K)=L$. 
It is clear that a compact covering map is surjective.
\begin{thm}\label{T:Choquet}
Let $X$ be a topological space. The following are equivalent:
\begin{enumerate}[\upshape (i)]
	\item $X$ is strongly $\mathcal{K}$--analytic.
	\item There exists a compact topological space $\Omega$, a strongly $\K$ subset $C$ of 
$\Omega$ and a compact covering map $f \colon C \to X$.
  \end{enumerate}
\end{thm}
\begin{proof}
\textbf{(ii) $\Rightarrow $ (i)} It is an immediate consequence of Prop. \ref{P:Kanalytic} and \cite[Prop. 3.4]{KM}

\textbf{(i) $\Rightarrow $ (ii)} From \cite[Prop. 3.2]{KM} 
$X$ is the image through a compact covering map $f$ of a closed subset $C$ of a space 
$M \times K$, where $M$ is a Polish space and $K$ is a compact space. Let $K_0$ be 
a compactification of $M$. From Th. \ref{T:compactification}  $M$ 
is a strongly $\K$ subset of  $K_0$, consequently by 
Prop. \ref{finiteprod} the space $M \times K$ is a strongly $\K$ subset of the compact space 
$\Omega =K_0 \times K$. It follows that $C$ is a strongly $\K$ subset of $\Omega$, as a closed subset of 
$M \times K$. 
(cf. Remarks \ref{R:Kcd} and \ref{R:Cech}.)
\end{proof}
The following result is analogous to \cite[Th. 1.12]{MS} and \cite[Th. 3.1]{KM}.
We recall that if $X$ is any topological space then the Vietoris topology on $\mathcal{K}(X)$ has a basis 
consisting of sets of the form 
\[\beta (V_1,\ldots, V_n)=\left\{K\in 
	\mathcal{K}(X): K\subseteq\bigcup_{i=1}^n V_i\; \text{and}\; K\cap V_i\neq
	\emptyset \; \text{for} \; i=1,\ldots,n\right\},
\]
where $n\in \mathbb{N}$ and $V_1,\ldots, V_n$ are open non empty subsets of $X$ 
(cf. \cite[p. 162]{E}).
\begin{thm}\label{T:Vietoris}
Let $X$ be a topological space and $A \subseteq X$. The following are equivalent:
\begin{enumerate}[\upshape (i)]
	\item $A$ is a strongly $\K$ subset of $X$.
	\item $\mathcal{K}(A)$ is a strongly $\K$ subset of $\mathcal{K}(X)$.
	\item $\mathcal{K}(A)$ is a $\K$ subset of $\mathcal{K}(X)$.
\end{enumerate}
\end{thm}
\begin{proof}
\textbf{(i)$\Rightarrow$(ii)} As $A$ is a strongly $\K$ subset of $X$, there exists a double sequence 
$K_{n,m}$, $n$, $m \in \mathbb{N}$ of compact subsets of $X$ such that 
$A=\bigcap_{n=1}^\infty \bigcup_{m=1}^\infty K_{n,m} $ and for each compact subset $L$ of
$A$ and for each $n\in \mathbb{N}$ there exists $m \in \mathbb{N}$, such that $L \subseteq K_{n,m}$. 
We shall prove that the double sequence $\mathcal{K}(K_{n,m})$, $n$, $m \in \mathbb{N}$ 
of compact subsets of $\mathcal{K}(X)$, makes $\mathcal{K}(A)$ into a strongly $\K$ subset of $\mathcal{K}(X)$.\\
Step 1. Claim:
$\mathcal{K}(A)=\bigcap_{n=1}^\infty \bigcup_{m=1}^\infty \mathcal{K}(K_{n,m})$. \\
Let $K \in \mathcal{K}(A)$. As $A$ is a strongly $\K$ subset of $X$, for each 
$n \in \mathbb{N}$ there exists $m \in \mathbb{N}$ such that $K \subseteq K_{n,m}$. 
In other words $K \in \mathcal{K}(K_{n,m})$, hence $K \in \bigcup_{m=1}^\infty \mathcal{K}(K_{n,m})$. Then
$\mathcal{K}(A) \subseteq \bigcap_{n=1}^\infty \bigcup_{m=1}^\infty \mathcal{K}(K_{n,m})$.
Assume $K \in \bigcap_{n=1}^\infty \bigcup_{m=1}^\infty \mathcal{K}(K_{n,m})$. For each 
$n\in \mathbb{N}$ there exists $m \in \mathbb{N}$, such that $K \in \mathcal{K}(K_{n,m})$, that is 
$K \subseteq K_{n,m}$, hence 
$K \subseteq \bigcap_{n=1}^\infty \bigcup_{m=1}^\infty K_{n,m}$, that is, $K \subseteq A$. Consequently 
$\bigcap_{n=1}^\infty \bigcup_{m=1}^\infty \mathcal{K}(K_{n,m}) \subseteq \mathcal{K}(A)$
and the proof of the claim is complete.\\
Step 2. Claim: 
For each $n \in \mathbb{N}$ and for each compact subset $L$ of $\mathcal{K}(A)$ there exists 
$m\in \N$ such that $L \subseteq \mathcal{K}(K_{n,m})$.\\
Let $n \in \mathbb{N}$ and $L$ a compact subset of $\mathcal{K}(A)$. Consider the set
$X_L=\bigcup L=\bigcup \{\, K \mid K \in L \, \}$ ,
that is a compact subset of $A$. Then there exists $m \in \mathbb{N}$ such that 
$X_L \subseteq K_{n,m}$. Thus for each $K \in L$ we have $K \subseteq K_{n,m}$, that is, $K \in \mathcal{K}(K_{n,m})$. 
Then $L \subseteq \mathcal{K}(K_{n,m})$ and the proof of the claim is complete.
Thus $\mathcal{K}(A)$ is a strongly $\K$ subset of $\mathcal{K}(X)$.\\
\textbf{(ii)$\Rightarrow$(iii)} It is obvious. \\
\textbf{(iii)$\Rightarrow$(i)} Let $\{\Omega_{n,m}:n,m \in \N \}$ be a double sequence of compact sets in 
$\mathcal{K}(X)$ such that $\mathcal{K}(A)=\bigcap_{n=1}^\infty\bigcup_{m=1}^\infty \Omega_{n,m}$; 
then setting $Y_s=\bigcap_{i=1}^n \Omega_{i,m_i}$ for $s=(m_1,m_2,\cdots m_n)$ we get that 
$\mathcal{K}(A)=\bigcup_\sigma \bigcap_n Y_{\sigma |n}$ and the map 
$\sigma \to \bigcap_{n=1}^\infty Y_{\sigma |n}$ is usco. For each $n$, $m \in \N$ 
the set $K_{n,m}=\bigcup\{K : K \in \Omega_{n,m} \}$ is compact in $X$. 
Now let $B=\bigcap_{n=1}^\infty\bigcup_{m=1}^\infty K_{n,m}$; it is enough to prove that $A=B$. 
It is obvious that $A\subseteq B$; let $x\in B$. Then there exist sequences 
$\sigma =(m_n) \in \Sigma$ and $(K_n) \subseteq \mathcal{K}(X)$ such that 
$x \in K_n$ and $K_n \in \Omega _{n, m_n}$ for all $n \in \N$. Then the sequence 
$(K_n)$ has a limit point, say $K \in \bigcap_{n=1}^\infty \Omega_{n, m_n}=\bigcap Y_{\sigma |n}$ 
(see \cite[Lemma 3.1]{AAM1}). Therefore $K$ is a compact subset of $A$ and $x \in K \subseteq A$.
\end{proof}
\begin{rem}
A $\K$ subset of a compact space is not necessarily a strongly $\K$ set. 
Indeed, let $X=[0, 1]$ and $A=\mathbb{Q} \cap X$; then $A$ is a (countable 
and hence) $\sigma$--compact set, but it is not a strongly 
$\mathcal{K}$--analytic space (see Prop. \ref{P:hemi} and \cite[Prop. 1.15]{MS}). 
Note that there exist countable topological spaces, which are not even 
strongly countably determined (see \cite[Remark 4.4]{KM} and \cite[Example 2.9]{MS}).
\end{rem}
\begin{quest}\label{Q:A}
Let $K$ be a compact space and $X \subseteq K $.
\begin{enumerate}[\upshape(i)]
	\item Assume that $X$ is a strongly $\mathcal{K}$--analytic subspace of $K$. Is then $X$ a strongly $\K$ or 
	at least a $\K$ subset of $K$? 
	\item Assume that $X$ is a strongly $\K$ subset of $K$. Is then $X$ strongly $\K$ in every compact space, which contains $X$ homeomorphically?
\end{enumerate}
\end{quest}
\begin{rem}
The analogous questions in the $\K$ case have both negative answers. 
Indeed, concerning the first question, consider a Borel (or an analytic non Borel) 
non $\K$ subset of $[0, 1]$. For the second, note that Talagrand (\cite{T2}) was the first 
who constructed a counterexample; still another counterexample can be found in \cite{AAM1}.
\end{rem}
\section{Banach spaces SWCG with respect to a superspace} \label{S:B}
We recall that a Banach space $X$ is called strongly weakly compactly generated (SWCG) if there exists a weakly compact subset $K$ of $X$ that strongly generates $X$, that is, for every weakly compact subset
 $L$ of $X$ and for every $\ee >0$ there exists
$n \in \mathbb{N}$ with $L \subseteq nK+\ee B_X$. The class of SWCG Banach spaces is introduced and studied 
by Schl\" uchtermann and Wheeler in \cite{SW}. In the next definition we generalize the concept 
of SWCG Banach space, considering the weakly compact subset $K$, that strongly generates $X$ in a superspace 
$Z$ of $X$. Thus, we are led to the following definition.
\begin{defin} \label{D:relSWCG}
Let $X$, $Z$ be Banach spaces with $X \subseteq Z$. We shall say that $X$ 
\emph{is strongly weakly compactly generated (SWCG) with respect to (or relatively to)} 
$Z$ if there exists a weakly compact (convex and symmetric) subset $K$ of $Z$, 
such that for each $\varepsilon > 0$ and for each  weakly compact 
subset $L$ of $X$ there exists $n \in \mathbb{N} $ with 
$L\subseteq nK+\varepsilon B_Z$. 
\end{defin}
\begin{rem} \label{R:SWCG}
\textbf{(i)} The previous definition is equivalent to the following: A Banach space $X$ is SWCG with respect to a superspace $Z$ if there exists a sequence $(K_n)$ of weakly compact, convex symmetric subsets of 
$Z$ such that for every weakly compact subset $L$ of $X$ and for every $\ee >0$ there exists 
$n \in \mathbb{N}$ with 
$L \subseteq K_n+\varepsilon B_Z$ (cf. \cite[Th. 2.1]{SW}).

\textbf{(ii)} Let $Z$ be a SWCG Banach space, according to the definition 
given in \cite{SW}, and let $K$ be a weakly compact set, that strongly generates $Z$. It is clear that if 
$X$ is a closed linear subspace of $Z$, then $K$ strongly generates $X$ according to Defin. \ref{D:relSWCG}. Moreover as there exists a closed linear subspace $X$ of $Z=L_1[0, 1]$, that is not SWCG 
(\cite[Cor. 3.10]{MS}, \cite{R1}), 
Defin. \ref{D:relSWCG} makes sense. If $X=Z$, then Defin. \ref{D:relSWCG} gives us the concept of the SWCG Banach space according to Schl\"uchterman and Wheeler (\cite{SW}).

\textbf{(iii)} Let $X$, $Z$ be Banach spaces such that $X \subseteq Z$ and $X$ is SWCG with respect to $Z$. 
If $Y$ is a closed linear subspace of $X$, then it is SWCG with respect to $Z$. Thus the class of spaces introduced by Defin. \ref{D:relSWCG} is closed under subspaces.
\end{rem}
\begin{prop} \label{P:WCG}
If a Banach space $X$ is SWCG with respect to a Banach space $Z$, then there exists a WCG Banach 
space $Y$, with $X \subseteq Y \subseteq Z$, such that $X$ is SWCG with respect to $Y$. 
\end{prop}
\begin{proof} Consider a weakly compact subset $K$ of $Z$, that strongly generates $X$. Put 
$Y= \overline{\langle K \rangle}$ (the closed linear hull of $K$ in $Z$). It is evident that $Y$ is a proper space.
\end{proof} 
\begin{rem}\label{R:relSWCG}
\textbf{(i)} Let $X \subseteq Z_1 \subseteq Z_2$ be Banach spaces. If $X$ is SWCG with respect to $Z_1$, then it is SWCG with respect to $Z_2$.

\textbf{(ii)} Let $X \subseteq Z$ be Banach spaces, such that $X$ is SWCG with respect to its superspace $Z$. 
By Prop. \ref{P:WCG}  $Z$ may be assumed WCG, consequently $X$ as a subspace of $Z$ is WKA. Later we shall prove that $X$ satisfies stronger conditions. More precisely, we shall prove that $X$ is SWKA (Cor. \ref{C:SWKA}) 
and moreover is a strongly $\K$ subset of $(X^{**}, w^*)$.
\end{rem}
In the sequel we prove that the above property remains invariant under isomorphisms. In fact, we prove something more general. This property is preserved by weakly--compact covering linear operators.
\begin{prop}\label{P:WCCSG}
Let $X$,$Y$, $Z$ be Banach spaces such that $X$ is SWCG with respect to $Z$. If there exists 
a weakly--compact covering  linear operator $T \colon X \to Y$, then there exists a Banach space $Z_1$, such that $Y$ is SWCG with respect to $Z_1$.
\end{prop}
\begin{proof}
The Banach space $Y$ is contained isometrically in the Banach space 
$\ell_\infty (\varGamma)$ for some set $\varGamma$. The space  $\ell_\infty (\varGamma)$ is injective, 
so there exists a bounded linear operator 
$\widetilde{T} \colon Z \to \ell_\infty (\varGamma)$, that extends $T$. Put $Z_1=\ell_\infty (\varGamma)$ 
and $\Omega=\widetilde{T}(K)$, where $K$ is any weakly compact subset of $Z$ that strongly generates $X$. 
We are going to prove that $\Omega$ strongly generates $Y$. Consider a weakly compact subset $L$ of $Y$ and $\varepsilon >0$. As the oparetor $T$ is weakly compact covering, there exists a weakly compact subset $L_1$ of $X$, such that $T(L_1)=L$. The set $K$ strongly generates $X$, hence there exists 
$n \in \mathbb{N}$ such that
$L_1 \subseteq nK+\frac{\varepsilon}{\|\widetilde{T}\|}B_Z$, hence 
$\widetilde{T}(L_1) \subseteq n \widetilde{T}(K)+\frac{\varepsilon}{\|\widetilde{T}\|}\widetilde{T}(B_Z)$.
Then
$L \subseteq n\Omega+\frac{\varepsilon}{\| \widetilde{T} \|}\| \widetilde{T} \|B_{Z_1}$, 
that is, $L \subseteq n\Omega +\varepsilon B_{Z_1}$.
Thus the weakly compact set $\Omega$ strongly generates $Y$, so $Y$ is SWCG with respect to $Z_1$.
\end{proof}
\begin{prop}
Let $X$, $Z$ be Banach spaces and $X$ is SWCG with respect to $Z$. If $Y$ is isomorphic to $X$, then there exists a  Banach space $Z_1$, such that $Y$ is SWCG with respect to $Z_1$.
\end{prop}
\begin{proof}
The proof is immediate by Prop. \ref{P:WCCSG}.
\end{proof}
The next result is the analogue of \cite[Th. 2.7]{SW} (and also of \cite[Prop.4.2]{KM}).
\begin{prop}
Let $X$, $Z$ be Banach spaces and $X$ is SWCG with respect to $Z$. If $Y$ is a reflexive subspace of 
$X$, then the space $X/Y$ is SWCG with respect to a Banach space $Z_1$.
\end{prop}
\begin{proof}
The natural map $\pi \colon X \to X/Y$ with $\pi (x)=x+Y$ is weakly--compact covering 
(cf. \cite[Prop. 4.2]{KM}). So our claim is immediate from Prop. \ref{P:WCCSG}.
\end{proof}
\begin{prop} \label{P: separable}
Let $X$, $Z$ be Banach spaces and $X$ is SWCG with respect to $Z$. If $X$ separable, then there exists a separable Banach space $Y$, such that $X$ is SWCG with respect to $Y$.
\end{prop}
\begin{proof}
As we proved above $X$ is SWCG with respect to a WCG Banach space, so we can assume that $Z$ is WCG. 
As $X$ is a separable subspace of a (non separable) WCG Banach space there exists a projection 
	\[P \colon Z \to Z \quad \text{with} \quad \| P\| =1\, , \quad P(Z) \quad \text{separable and} \quad X \subseteq P(Z)\, .
\]
(cf. \cite[Th. 3.1]{L}).
We are going to prove that $X$ is SWCG with respect to the separable Banach space $Y=P(Z)$. Consider a weakly compact subset $K$ of $Z$, that strongly generates $X$ and put $\Omega =P(K)$. 

Let $L$ be a weakly compact subset of $X$ and $\ee >0$. Then there exists $n \in \mathbb{N}$ such that 
$L \subseteq nK+\ee B_Z$, hence $P(L) \subseteq nP(K)+\ee P(B_Z)$. As $P$ is a projection with 
$\| P \| =1$ and $L \subseteq X \subseteq P(Z)$, by the last relation follows 
$L \subseteq n\Omega+\ee B_Y$. Therefore $X$ is SWCG with respect to the separable Banach space $Y$.
\end{proof}
\begin{rem}
The separable space $Y$ of Prop. \ref{P: separable} can be chosen to be the space $C(B_{X^*})$. 
Indeed, $Y$ is embedded into $C(B_{Y*})$, which in its turn is isomorphic (by Milyutin's theorem) 
to the space $C(B_{X*})$. We do not know whether this result holds true in the case where 
$X$ is not separable.
\end{rem}
In the sequel we state a characterization of the class introduced by Defin. \ref{D:relSWCG}, 
which is analogous to the characterization of the class of SWCG Banach spaces (\cite[Th. 2.1]{SW}. 
Since its proof is analogous to the corresponding proof for SWCG spaces, we give a brief outline of it.
\begin{thm} \label{T:Mackey}
Let $X$ be a Banach space. The following are equivalent:
\begin{enumerate}[\upshape(i)]
	\item $X$ is SWCG with respect to a Banach space.
	\item $X$ is contained in a Banach space $Z$ and exists a metrizable topology $\tau _d$ in 
	$B_{Z^*}$ such that:			
\begin{enumerate}
	\item [(a)] The Mackey topology $\tau$ on $B_{Z^*}$ is finer than $\tau_d$.
	\item [(b)] The topology $\tau_d$ is finer than topology $\tau_X$ on $B_{Z^*}$ of uniform convergence 
	on weakly compact subsets of $X$.
\end{enumerate}
\end{enumerate}
\end{thm}
 \begin{proof}
\textbf{(i)$\Rightarrow$ (ii)} Let $X\subseteq Z$ and $K \subseteq Z$ be a 
weakly compact convex and symmetric set that strongly generates $X$. A metric $d$ whose  
topology $\tau_d$ satisfies our requirements is the following
\[d(x^*, y^*)=\max \{|x^*(x)-y^*(x)|: x\in K \} \quad \text{for all} \quad x^*,y^* \in B_{Z^*}\, .
\]

\textbf{(ii) $\Rightarrow$ (i)} Let $Z$ be a Banach space with 
$X \subseteq Z$, that satisfies conditions 
(a) and (b). Consider a neighbourhood base $(U_n)$ of $0$ for $(B_{Z^*}, \tau _d)$. 
As the Mackey topology $\tau$ is finer than the metric topology $\tau _d$ 
there exists a sequence $(K_n)$ of weakly compact 
convex symmetric subsets of $Z$ such that 
$ K_n^0\cap B_{Z^*} \subseteq U_n$ for all $n \in \mathbb{N}$. 
Let $L$ be a weakly compact subset of $X$ and $0<\ee <1$. Put $c=\dfrac{1}{\ee}$. 
As the metric topology 
$\tau _d$ is finer than $\tau _X$ there exists $n \in \mathbb{N}$ such that 
$U_n \subseteq (cL)^0 \cap B_{Z^*}$. 
We continue as in the proof of implication (a)$\Rightarrow$(b) of \cite[Th. 2.1]{SW}.
\end{proof}
\begin{rem}
The metric $d$, that appears in the proof of direction (i)$ \Rightarrow $ (ii) of 
Th. \ref{T:Mackey} is nothing else but the metric of supremum norm of the space $C(K)$. Indeed, the 
operator $T \colon Z^* \to C(K)$ defined by $T(z^*)=z^*|K$ is bounded, linear and 1--1 
(as $K$ can be assumed total in $Z$). 
Moreover $T$ is $(w^*, \tau _p)$ continuous (where $\tau _p$ is the topology of pointwise convergence), so  $\Omega=T(B_{Z^*})$ is $\tau _p$--compact set homeomorphic with the unit ball $(B_{Z^*}, w^*)$. 
It is also bounded, so Grothendieck's theorem implies that it is weakly compact. 
Therefore the space $\Omega$ endowed with the metric of 
supremum norm is a complete metric space, as a closed subset of the Banach space $C(K)$.
\end{rem}
\begin{cor}
Let $X\subseteq Z$ be Banach spaces and $X$ is SWCG with respect to $Z$. If $X$ is separable, then the space 
$(B_{X^*}, \tau)$ is analytic, where $\tau$ is the Mackey topology on $X^*$.
\end{cor}
\begin{proof}
We can assume and we do that $Z$ is separable, so by the previous remark $C(K)$ is a separable 
Banach space (as $K$ is weakly compact metrizable set). It follows that the closed subspace 
$\Omega =T(B_{Z^*})$ of $C(K)$ endowed with the metric $d$ of supremum norm is 
a complete separable metric space. By assertion (b) of (ii) of Th. \ref{T:Mackey} the map 
	\[\Phi \colon (B_{Z^*}, \tau _d) \to (B_{X^*}, \tau) \quad \text{with} \quad \Phi (z^*)=z^*|X
\] 
is continuous and onto, hence the conclusion follows.
\end{proof}
\begin{rem}
A separable Banach space $X$ such that the space $(B_{X*}, \tau)$ is analytic is not necessarily 
SWCG with respect to some Banach space. Indeed, consider a non reflexive Banach space $X$ 
with separable dual. Then by Cor. \ref{C:reflex} (or \ref{C:Asplund}) $X$ does not satisfy 
Def. \ref{D:relSWCG}; on the other hand since the norm topology is finer than the Mackey 
topology on $B_{X*}$ we get that $(B_{X*}, \tau)$ is an analytic space.
\end{rem}
 
As we have noted (Rem. \ref{R:SWCG} ) a closed linear subspace $X$, of a SWCG Banach space $Z$, is 
SWCG with respect to $Z$. But we do not know the answer to the next question.
\begin{quest}
Let $X \subseteq Z$ be Banach spaces and $X$ is SWCG with respect to $Z$. Is then $X$ isomorphic to a subspace of a 
SWCG Banach space? In other words can we always assume in Def. \ref{D:relSWCG} that $Z$ is SWCG? 
(According to Prop. \ref{P:WCG} we can always assume that $Z$ is WCG.)
\end{quest}
 
In the sequel, we are going to prove that the class defined by Def. \ref{D:relSWCG} is contained in the class of 
SWKA Banach space. We generalize Def. \ref{D:relSWCG} and we shall prove that the (at least formally) larger class has the same property. Firstly, we recall an internal characterization of subspaces of WCG Banach spaces, due to 
Fabian, Montesinos and Zizler (\cite[Th. 1]{FMZ1}).
Namely, it is proved that the following significant result holds:
\begin{xthm}[Fabian, Montesinos, Zizler] 
A Banach space $X$ is a subspace of a WCG Banach space if and only if for each $p \in \mathbb{N}$ there exists a sequence $(M_{n,p})_n$ consisting of $\dfrac{1}{p}$--weakly relatively compact subsets of $X$ such that  $X=\bigcup_{n=1}^\infty M_{n,p}$.
\end{xthm}

A natural question which arises from the previous theorem is whether 
we can give an analogous characterization for subspaces of SWCG Banach spaces. 
Presently, we do not have an answer to this question, but a necessary condition 
for a space to be SWCG with respect to a Banach space, according to 
Def. \ref{D:relSWCG} is given by the following proposition. We note that in 
this proposition we follow the method of proof of \cite[Th. 1]{FMZ1}.
\begin{prop} \label{P:subspace}
If the Banach space $X$ is SWCG with respect to a Banach space $Z$, then there exists a family 
$\{M_{n,p}:n,p \in \mathbb{N}\}$ of subsets of $X$ with the following properties:
\begin{enumerate}[\upshape (i)]
	\item For every $n \in \mathbb{N}$ and for every $p \in \mathbb{N}$ the set $M_{n,p}$ is $\frac{1}{p}$--weakly relatively compact.
	\item For every $p\in \mathbb{N}$ and for every weakly compact subset $L$ of $X$ there exists 
	$n \in \mathbb{N}$ such that $L \subseteq M_{n,p}$.
\end{enumerate}
\end{prop}
\begin{proof}
Let $K$ be a weakly compact subset of $Z$, that strongly generates $X$. Then for each $n$, 
$p \in \mathbb{N}$ the set $nK+\dfrac{1}{p}B_Z$ is $\frac{1}{p}$--weakly relatively compact subset of $Z$. Indeed,
	\[\overline{nK+\frac{1}{p}B_Z}^*  \subseteq \overline{nK}^* +\overline{\frac{1}{p}B_Z}^*
															=nK+\frac{1}{p}B_{Z^{**}}
															\subseteq Z+\frac{1}{p}B_{Z^{**}}.
\]
For every $n$, $p \in \mathbb{N}$ put 
$M_{n,p}=X\cap(nK+\frac{1}{4p}B_Z) $. 
Then, by \cite[Prop. 3.62]{HMVZ}, the set $M_{n,p}$ is an $\frac{1}{p}$--weakly relatively compact 
subset of $X$ for every $n$, $p \in \mathbb{N}$. Additionally, for every weakly compact subset 
$L$ of $X$ and for each $p \in \mathbb{N}$ there exists $n \in \mathbb{N}$ such that $L \subseteq nK+\dfrac{1}{4p}B_Z$, so $L \subseteq M_{n,p}$. Therefore the family $\{M_{n,p} :n,p \in \mathbb{N} \}$ has the desired properties.
\end{proof}

So we arrive at the following definition.
\begin{defin} \label{D:*}
We shall say that a Banach space $X$ has \emph{property (*)}, if there exists a countable family 
$\{M_{n,p} :n, p \in \mathbb{N} \}$ of (convex and symmetric) subsets of $X$, that satisfies assertions 
(i) and (ii) of Prop. \ref{P:subspace}.
\end{defin}

It is clear, from Prop. \ref{P:subspace} that every Banach space $X$ that is 
SWCG with respect to a Banach space $Z$ has property (*). If $A$, $B$ are 
$\ee$--weakly relatively compact subsets of a Banach space $X$, then 
clearly the set $A \cup B$ is $\ee$--weakly relatively compact. Thus in Def. \ref{D:*} 
we can assume that $M_{n, p} \subseteq M_{n+1, p}$ for each $n$, $p \in \mathbb{N}$. 
Moreover we can assume that the sets $M_{n,p}$ are convex and symmetric 
(see \cite[Th. 3.64]{HMVZ}).
\begin{rem} \label{R:*}
\textbf{(a)} If a Banach space $Z$ has property (*), then every closed linear subspace of 
$X$ has property (*). Indeed, if a countable family $\{M_{n, p}: n, p \in \N \}$ 
of subsets of $Z$ witnesses that $Z$ has property (*), then by \cite[Prop. 3.62]{HMVZ} the sets 
$K_{n, p}=X\bigcap M_{n, 4p}$\, , $n$, $p \in \N$ ensure that $X$ has property (*).

\textbf{(b)} It is easy to see that a Banach space $X$ has property (*) if there exists 
a countable family $\{M_{n, p} :n, p \in \mathbb{N} \}$ of ( convex and symmetric) 
subsets of $B_X$ such that assertions (i) and (ii) of Prop. \ref{P:subspace} 
are satisfied for $B_X$.
\end{rem}
\begin{thm}\label{T:*SWKA}
Let $X$ be a Banach space with property (*). Then $X$ is SWKA and subspace of a WCG Banach space.
\end{thm}
\begin{proof}
Let $\{M_{n,p} :n, p \in \mathbb{N} \}$ be a family of subsets of $X$ that 
witnesses property (*) of $X$.

The fact that the space $X$ is a subspace of a WCG Banach space follows immediately by \cite[Th. 1]{FMZ1}.
To show that $X$ is SWKA consider the map 
$F \colon \varSigma \to \mathcal{K}(X)$ with $F(\sigma)=\bigcap _{p=1}^\infty 
\overline{M}_{\sigma(p), p}^*$.
It is not difficult to see that $F$ is well defined and usco.
The proof is similar to the proof of direction 
(ii)$\Rightarrow$ (iii) of \cite[Prop. 3.1]{KM}, so we omit it.
\end{proof}
\begin{cor} \label{C:SWKA}
Every Banach space $X$, that is SWCG with respect to a Banach space $Z$ is SWKA.
\end{cor}
\begin{proof}
It is immediate from Th. \ref{T:*SWKA} and Def. \ref{D:*}.
\end{proof}
The next result extends a result of Schl\"uchterman and Wheeler (\cite[Th. 2.5]{SW}) 
in the wider class of Banach spaces with property (*).
\begin{thm} \label{T:wsc}
Let $X$ be a Banach space with property (*). Then $X$ is weakly sequentially complete.
\end{thm}
\begin{proof}  
Let $\{K_{m, p} : m, p \in \mathbb{N} \}$ be a family of closed 
convex symmetric subsets of $X$ that witnesses property (*) and 
$K_{m, p} \subseteq K_{m+1, p}$ for all $m$, $p \in \mathbb{N}$.

Let $(x_n)$ be a weakly Cauchy sequence of $X$. Then there exists $x^{**} \in X^{**}$ such that 
$x_n \stackrel{w^*}{\rightarrow} x^{**}$. Assume that the set 
$\{ x_n : n \in \mathbb{N}\}$ is not weakly relatively compact. The family 
$\{ 2K_{m, p} : m, \, p \in \mathbb{N}\}$ consists of bounded, convex symmetric subsets and satisfies 
the following conditions:
\begin{enumerate}
	\item [(a)] For every $m$, $p$ the set $2K_{m,p}$ is $\frac{2}{p}$--weakly relatively compact.
	\item [(b)] For every $m$, $p \in \mathbb{N}$ holds $2K_{m, p} \subseteq 2K_{m+1, p}$.
	\item [(c)] For every weakly compact subset $L$ of $X$ and for every $p \in \mathbb{N}$ there exists 
	$m \in \mathbb{N}$ such that $L \subseteq 2K_{m,p}$.
\end{enumerate}	
Then there exists $p \in \mathbb{N}$ such that 
$\{ x_n : n \in \mathbb{N} \} \nsubseteq 2K_{m, p}$ for all $m \in \mathbb{N}$. 
Thus for every $m \in \mathbb{N}$ there exists $n_m \in \mathbb{N}$ such that $x_{n_m} \notin 2K_{m, p}$. 
Moreover for each $m \in \mathbb{N}$ infinitely many terms of the sequence do not belong to $2K_{m,p}$. 
For every $n \in \mathbb{N} $ define	$m_1(n)= \min \{m \in \mathbb{N} : x_n \in 2K_{m, p} \}$
and	$m_2(n)= \min \{m \in \mathbb{N} : x_n \in K_{m, p} \}$.
Clearly $m_1(n) \leq m_2(n)$ for every $n \in \mathbb{N}$, $x_n \in 2K_{m_1(n), p}$ and 
$x_n \notin 2K_{\lambda, p}$ if and only if $\lambda < m_1(n)$.

We can choose inductively a subsequence $(x_{k_n})$ of $(x_n)$ such that (1)
$m_1(k_{n+1}) > m_2(k_{n})$ for all $n \in \mathbb{N}$. 
Choose $k_1 =1$. Assume that $k_1 < k_2< \cdots <k_n$ have been chosen. For 
$m=m_2(k_n)$ there exists $n_m$ such that $x_{n_m} \notin 2K_{m, p}$ and $n_m>k_n$. 
Put $k_{n+1}=n_m$ and then we have $m_1(k_{n+1})>m=m_2(k_n)$.
The sequence $(x_{k_n})$ is weakly Cauchy, hence there exists $m_0 \in \mathbb{N}$ such that 
$ x_{k_{n+1}} - x_{k_n} \in K_{m_o, p}$  for all $n \in \mathbb{N}$. 
Then it follows $x_{k_{n+1}} =(x_{k_{n+1}}-x_{k_n})+x_{k_n} \in K_{m_o, p}+K_{{m_2}(k_n), p}$ 
for all $n \in \mathbb{N}$.
Choose $m_2(k_n) \geq m_o$. Then 
$K_{m_o, p}+K_{m_2(k_n), p} \subseteq 2K_{m_2(k_n), p}$, so 
$x_{k_{n+1}} \in 2K_{m_2(k_n), p}$;
consequently 
$m_1(k_{n+1}) \leq m_2(k_n)$, which contradicts (1). 
We conclude then that the set $\{ x_n : n \in \mathbb{N} \}$ is weakly relatively compact, 
hence $x^{**} \in X$ and the space $X$ is weakly sequentially complete.
\end{proof}
\begin{cor} \label{C:wsc}
Every Banach space $X$ that is SWCG with respect to a Banach space $Z$ is weakly sequentially complete.
\end{cor}
\begin{proof}
It is immediate from Th. \ref{T:wsc}, as $X$ has property (*).
\end{proof}
\begin{cor} \label{C:reflex}
Let $X$ be a Banach space not containing $\ell_1(\mathbb{N})$. Then $X$ has property (*) 
if and only if it is reflexive.
\end{cor}
\begin{proof}
If $X$ is reflexive, then its closed unit ball $B_X$ is a weakly compact set, 
hence $X$ is SWCG and consequently has property (*).

Assume that $X$ has property (*). Then according to Th. \ref{T:wsc}  $X$ is weakly 
sequentially complete and as $\ell_1(\mathbb{N})$ is not contained in $X$, 
by Rosenthal's $\ell_1$-theorem follows that $X$ is reflexive.
\end{proof}
\begin{cor} \label{C:Asplund}
Let $X$ be an Asplund Banach space (In particular $X$ has separable dual). If $X$ has property (*), 
then it is reflexive.
\end{cor}
\begin{proof}
Every separable subspace of $X$ has separable dual, so $X$ does not contain $\ell_1(\mathbb{N})$. 
According to Cor. \ref{C:reflex} $X$ is reflexive.
\end{proof}

We recall that a Banach space is called Polish if its closed unit ball $(B_X, w)$ is a Polish topological space. 
We also say that a Banach space $X$ has \v{C}ech complete ball if $B_X$ is a 
$G_\delta$ subset of $(B_{X^{**}}, w^*)$. These classes of Banach spaces are introduced and 
studied by Edgar and Wheeler in \cite{EW} (see also \cite{R2}).
\begin{cor} \label{C:Cech}
Let $X$ be a non reflexive Banach space with \v{C}ech complete ball $(B_X, w)$ (In particular, $X$ 
is a non reflexive Polish Banach space). Then $X$ does not have property (*).
\end{cor}
\begin{proof}
Every Banach space with \v{C}ech complete ball is Asplund (\cite{EW}).
So the result follows from Cor. \ref{C:Asplund}.
\end{proof}
It is known that an $\ell_p$--direct sum ($1 \leq p < +\infty $) 
of SWKA Banach spaces is SWKA (cf. \cite{KM}). We shall show now that 
property (*) is not preserved by $\ell_p$--direct sums ($p>1$), 
even if each space is SWCG.

\begin{thm} \label{T:sum}
The Banach space
\[
X=\left (\sum_{m=1}^\infty\oplus \ell_1 (\mathbb{N} \times \{ m \} )\right )_2
\] 
does not have property (*), although it is SWKA and weakly sequentially complete. In particular it is not SWCG 
with respect to any Banach space.
\end{thm}
\begin{proof}
The space $X$ is by \cite[Prop. 5.1]{KM} SWKA, as $\ell_1(\mathbb{N})$ is SWCG. Furtheremore  
	\[X=Y^* \quad \text{with} \quad 
	Y=\left (\sum_{m=1}^\infty\oplus c_0 (\mathbb{N} \times \{ m \} )\right )_2
\]
and $X$ has as an unconditional basis, the double sequence   
	\[e_{(n, m)}=(\underbrace{0,\ldots ,0}_{m-1},\bar{e}_{(n, m)},0, \ldots ,0) \quad n,m \in \N,
\]
where 
$\bar{e}_{(n, m)}$, $n \in \N$ is the usual basis of $\ell_1(\N)=\ell_1(\N \times \{m\})$.
(In the sequel we identify $\bar{e}_{(n, m)} \in \ell_1(\N)$ with $e_{(n, m)} \in X$.) 
The space $X$ does not contain $c_0(\N)$, as it is SWKA 
(cf. \cite[Prop. 1.9, Cor 1.10]{MS} (ii)), thus 
it is weakly sequentially complete (\cite[Th. 1.c.10]{LT}).

Assume, towards a contradiction that $X$ has property (*). Then there exists a family 
$\{M_{n,p}: n,p \in \mathbb{N} \}$ consisting of closed, convex and symmetric subsets of 
$B_X$ such that 
\begin{enumerate}[\upshape (i)]
	\item For every $n$, $p \in \mathbb{N}$ the set $M_{n,p}$ is $\frac{1}{p}$--weakly relatively compact.
	\item For every $p \in \mathbb{N}$ and for every weakly compact subset $L$ of $B_X$ 
	there exists $n \in \mathbb{N}$ with $L \subseteq M_{n,p}$.
\end{enumerate}
For every $\sigma \in \varSigma$ consider the space
	\[X_\sigma=\left( \sum_{m=1}^\infty \oplus \ell _1(\{1,2, \ldots \sigma (m)) \right)_2 \, ,
\]
which is reflexive and satisfies $X_\sigma \subseteq X$. By  $B_\sigma$ we denote the closed 
unit ball of $X_\sigma$. Then $\bigcup_{\sigma \in \Sigma} B_\sigma \subseteq B_X$ and for every 
$\sigma \in \varSigma$ the set $B_\sigma$ is weakly compact.

Let $p>2$. For every $n \in \mathbb{N}$ put $A_n =\{\sigma \in \varSigma : B_\sigma \subseteq M_{n,p} \}$; 
then we have $\varSigma =\bigcup_{n=1}^\infty A_n$.
Using a Baire category argument, due to Talagrand (cf. \cite[Th. 4.3]{T1}) we find $n_0 \in \mathbb{N}$ 
and an infinite subset $D=\{\sigma_k :k \in \mathbb{N} \}$ of $A_{n_0}$ such that for some 
$s_0 \in S$ holds 
$s_0< \sigma _k$ and $\sigma _k(m_0+1)=k$ for all $k \in \mathbb{N}$, 
where $m_0=|s_0|$= the length of the finite sequence $s_0$. As $D \subseteq A_{n_0}$, it follows that 
$B_{\sigma _k} \subseteq M_{n_0,p}$ for every $k \in \mathbb{N}$. Consequently 
\[
	\{e_{(n, m_0+1)} :1\leq n \leq k = \sigma_k (m_0+1)\} \subseteq B_{\sigma_k} \subseteq M_{n_0,p} \quad 
	\text{for every} \quad k \in \mathbb{N} \, .
\]
Thus, we conclude that 
$M=\{e_{(n, m_0+1)} : n \in \N \} \subseteq M_{n_0, p}$.
But this means that the usual basis of $\ell_1(\N)$ is $\frac{1}{p}$--weakly relatively compact 
subset of $X$, so $\frac{2}{p}$--weakly relatively compact subset of $\ell_1(\N)$, which is false 
as $\frac{2}{p} <1$ and the usual basis of $\ell_1(\N)$ is not $\ee$--weakly relatively compact 
subset of $\ell_1(\N)$ for $0< \ee <1$.
\end{proof}
\begin{rem} \label{R:sum}
\textbf{(a)} The previous result provides us with still another example with the properties 
of \cite[Example 2.6]{SW}; that is a weakly sequentially complete and separable space 
that is not SWCG. Note that as it has been proved in \cite[Example 2.9]{MS} the example in 
\cite{SW}, in contrast with the present example, is not a SWKA space (see also \cite[Remark 4]{KM}).

\textbf{(b)} The result described in Th. \ref{T:sum} is generalized, with the same essentially proof 
for a direct sum of the form 
	\[X=\left (\sum_{m=1}^\infty\oplus \ell_1 (\mathbb{N} \times \{ m \} ) \right )_p , \quad 
	\text{where} \quad 1<p<+\infty \, .
\]
\end{rem}
\begin{cor} \label{C:lN}
Let $(X_n)$ be a sequence of Banach spaces such that $\ell_1(\N)$ embeds in $X_n$ 
for infinitely many $n \in \N$. Then the space 
	\[X=\left (\sum_{n=1}^\infty\oplus X_n \right )_p , \quad 
	\text{where} \quad 1<p<+\infty \, ,
\]
does not have property (*).
\end{cor}
\begin{proof}
It follows immediately from Th. \ref{T:sum} and Remark \ref{R:sum}.
\end{proof}
\begin{cor} \label{C:sum}
Let $(X_n)$ be a sequence of Banach spaces, such that $X_n$ has property (*) for every 
$n \in \N$. If $X_n$ is not reflexive for infinitely many $n \in \N$, then the space 
	\[X=\left (\sum_{n=1}^\infty\oplus X_n \right )_p , \quad 
	\text{where} \quad 1<p<+\infty 
\]
does not have property (*).
\end{cor}
\begin{proof}
If $X_n$ is not reflexive, then from Cor.\ref{C:reflex}, it contains $\ell_1(\N)$. 
By Cor. \ref{C:lN} the conclusion follows.
\end{proof}
\begin{cor}
Let $(X_n)$ be a sequence of SWCG Banach spaces, such that $X_n$ is not reflexive for infinitely many 
$ n \in \N $. Then the space
\[X=\left (\sum_{n=1}^\infty\oplus X_n \right )_p , \quad 
	\text{where} \quad 1<p<+\infty \, 
\]
is not isomorphic to any subspace of a SWCG Banach space.
\end{cor}
\begin{proof}
It is clear by Cor. \ref{C:reflex} and Cor. \ref{C:sum}.
\end{proof}

In contrast to Th. \ref{T:sum} an $\ell_1$--direct sum of a sequence $(X_n)$ of SWCG Banach spaces is 
SWCG (\cite[Prop. 2.9]{SW}). An analogous result can be proved if every $X_n$ is SWCG 
with respect to some superspace of it.
\begin{prop} \label{P:prodrel}
Let $(X_n)$, $(Z_n)$ be sequences of Banach spaces, such that $X_n \subseteq Z_n$, 
for all $n \in \N$ and every 
$X_n$ is SWCG with respect to $Z_n$. Then the space 
$X=(\sum_{n=1}^\infty \oplus X_n)_1 $ 
is SWCG with respect to the space 
$Z=(\sum_{n=1}^\infty \oplus Z_n)_1 $.
\end{prop}
\begin{proof}
According to Th. \ref{T:Mackey} for every $n \in \N$ there exists a metrizable topology
$\tau _{d_n}$ in $B_{Z_n^*}$ such that:
\begin{enumerate}[\upshape (i)]
	\item $\tau _{d_n}$ is coarser than the Mackey topology $\tau _n$ on $B_{Z_n^*}$.
	\item $\tau _{d_n}$ is finer than the topology $\tau _{X_n}$ on $B_{Z_n^*}$ 
	(of uniform convergence on weakly compact subsets of $X_n$).
\end{enumerate}
As $Z$ is an $\ell_1$--direct sum of Banach spaces it follows that  
$B_{Z^*}=\prod_{n=1}^\infty B_{Z^*_n}$. Furthermore it is known by classical results about 
Mackey topology (cf. \cite[Prop. 1.2]{SW}) that Mackey topology $\tau$ on $B_{Z^*}$ is identified with the product topology of the spaces $(B_{Z_n^*}, \tau _n)$. Consider $B_{Z^*}$ endowed with the product topology 
of the topologies $\tau_{d_n}$, which is denoted by $\tau _d$. It is clear that $\tau _d$ 
is metrizable and coarser than $\tau$ (by (i)). If $\tau _X$ is the topology on $B_{Z^*}$ 
of uniform convergence on weakly compact subsets of $X$, then we can easily prove that 
$\tau _X$ is identified with the product topology of $\tau _{X_n}$, $n \in \N $, thus it is coarser 
than $\tau _d$ (the proof is analogous to the proof of the fact that the space $(B_{Z^*}, \tau)$ 
is identified with the topological product of the spaces $(B_{Z_n^*}, \tau _n)$). It follows 
from Th. \ref{T:Mackey} that $X$ is SWCG with respect to $Z$.
\end{proof}
\begin{rem}
It is worth noting that results similar to those described in the above proposition or the analogue of
Schl\" uchterman and Wheeler, are due to the fact that $\ell_1$ norm essentially does not introduce 
new compact sets in an $\ell_1$--direct sum of Banach spaces. So (with analogous techniques) 
we can prove that an $\ell_1$--direct sum of Banach spaces with property (*), has property (*) as well.
\end{rem}
\begin{quest}
Let $X$ be a Banach space with property (*). Is then $X$ SWCG with respect to some superspace?
\end{quest}

\section{Strongly $\K$ Banach spaces.}\label{S:C}
In this section we introduce a subclass of SWKA Banach spaces, that is, those Banach spaces 
which are strongly $\K$ subsets of $(X^{**}, w^*)$. As we shall see this subclass 
contains any other subclass of SWKA Banach spaces considered in this article. 
We recall that the class of $\K$ Banach spaces $X$ 
(that is, Banach spaces $X$, which are $\K$ subset of $(X^{**}, w^*)$) 
contains the closed subspaces of WCG Banach spaces (\cite{T1}) and is contained properly in 
WKA Banach spaces (\cite{AAM1}).
Let $X$ be a SWCG Banach space and $K$ a weakly compact (convex and symmetric) subset of $X$, 
that strongly generates $X$, that is, for every $\varepsilon >0$ and for every weakly compact subset 
$L$ of $X$ there exists $n \in \mathbb{N}$, such that $L \subseteq nK+\varepsilon B_X$. Then  
	\[
	X=\bigcap_{p=1}^\infty \bigcup _{m=1}^\infty \left ( mK+\frac{1}{p}B_{X^{**}} \right ) \, ,
\]
that is, $X$ is a $\K$ subset of $(X^{**}, w^*)$. Moreover for each $p \in \mathbb{N}$ and for each 
weakly compact subset $L$ of $X$ there exists $m \in \mathbb{N}$ such that 
\[
L \subseteq mK+ \frac{1}{p}B_{X^{**}}\, ,
\]
which means that every SWCG Banach space is a strongly $\K$ subset of its second dual $(X^{**}, w^*)$.

As shown below the same is true and for those Banach spaces which have property (*) 
(consequently for those Banach spaces which are SWCG with respect to some superspace). 
This leads to the following definition.
\begin{defin}
A Banach space $X$ is called \emph{strongly $\K$} if it is a strongly $\K$ subset 
of its second dual $(X^{**}, w^*)$.
\end{defin}
It is clear that every strongly $\K$ Banach space is SWKA (see Prop. \ref{P:Kanalytic}).
\begin{thm}
\label{T:*Kcd}
Every Banach space $X$ with property (*) is strongly $\K$.
\end{thm}
\begin{proof} 
Let $\{K_{n,p} : n, p \in \mathbb{N} \}$ be a family of convex symmetric subsets of $X$ 
with the following properties:
\begin{enumerate}[\upshape (i)]
	\item For every $n$, $p \in \mathbb{N}$ the set $K_{n,p}$ is $\frac{1}{p}$--weakly relatively compact.
	\item For every $p \in \mathbb{N}$ and for every weakly compact subset $L$ of $X$ 
	there exists	$n \in \mathbb{N}$ such that $L \subseteq K_{n,p}$.
\end{enumerate}
For every $n$, $p \in \mathbb{N}$ we have 
	\[
	\overline{K}^*_{n,p} \subseteq X+\frac{1}{p}B_{X^{**}} \, ,\quad \text{so} \quad
	X\subseteq \bigcup_{n=1}^\infty K_{n,p} \subseteq 
	\bigcup _{n=1}^\infty \overline {K}^*_{n,p} 
	\subseteq X+\frac{1}{p}B_{X^{**}}\, ,
\]
hence
	\[
	X \subseteq 
	\bigcap_{p=1}^\infty \bigcup_{n=1}^\infty \overline{K}^*_{n,p} \subseteq 
	\bigcap_{p=1}^\infty \left(X+\frac{1}{p}B_{X^{**}} \right) \subseteq X \, .
\]
Consequently 
	\[
	X=\bigcap_{p=1}^\infty \bigcup_{n=1}^\infty \overline{K}^*_{n,p}\, .
\]
Moreover for every $p \in \mathbb{N}$ and for every weakly compact subset $L$ 
there exists $n\in\mathbb{N}$ such that $L \subseteq \overline{K}^*_{n,p}$, hence 
$X$ is a strongly $\K$ Banach space.
\end{proof} 
Another class of strongly $\K$ Banach spaces is the class of Polish Banach spaces. 
\begin{lem} \label{L:w*}
Let $X$ be a separable Banach space. Then:
\begin{enumerate}[\upshape (i)]
	\item Every open (or closed) subset of $(X^*, w^*)$ is in $w^*$ topology a hemicompact space, 
	so a strongly $\K$ subset of $(X^*, w^*)$.
	\item Every $G_\delta$ subset of $(X^*, w^*)$ is a strongly $\K$ subset of $(X^*, w^*)$.
	\item Every strongly $\mathcal{K}$--analytic subset of $(X^*, w^*)$ is a $G_\delta$ 
	(hence a strongly $\K$) subset of $(X^*, w^*)$.
\end{enumerate}
\end{lem}
\begin{proof}
\textbf{(i)} As $(X^*, w^*)$ is a hemicompact space it follows immediately that every closed subset 
of it is also a hemicompact space.

Let $U$ be an open subset of $(X^*, w^*)$. For every $n \in \N$ put
	\[U_n=U \bigcap nB_{X^*}= U \bigcap B_{X^*}[0, n].
\]
It is clear that $U_n$ is an open subset of the compact metric space $B_{X^*}[0, n]$ 
(since $X$ is separable). Thus $U_n$ is a locally compact separable metric space, consequently it is a 
hemicompact topological space (cf. Example \ref{hemicompact}). 
$U_n$ can be written in the form $U_n=\bigcup _{m=1}^\infty K_{n, m}$, where every 
$K_{n, m}$ is a compact subset of $(X^*, w^*)$ and in addition for every 
$w^*$ compact subset $K$ of $U_n$ there exists 
$m \in \N$ such that $K \subseteq K_{n, m}$. We remark that 
	\[U=\bigcup _{n=1}^\infty U_n =\bigcup _{n=1}^\infty \bigcup _{m=1}^\infty K_{n, m} \, ,
\]
that is, $U$ is written as a countable union of compact subsets of it. Let $K$ be a $
w^*$ compact subset of $U$. Since $K$ is bounded it is contained in a closed ball 
$B_{X^*}[0, n]$, hence $U \subseteq U_n$. Then there exists $m \in \N$, such that 
$K \subseteq K_{n, m}$, which means that the countable family of compact sets 
$\{K_{n, m}:n, m \in \N \}$ dominates the compact subsets of $U$, thus $U$ is in 
$w^*$ topology a hemicompact space.

\textbf{(ii)}  By (i) and by Prop. \ref{P:intersection}, the conclusion follows.

\textbf{(iii)} Let $A$ be a strongly $\mathcal{K}$--analytic subset of $(X^*, w^*)$. 
Then each $A_n=A\bigcap nB_{X^*}$ is a strongly $\mathcal{K}$--analytic subset of 
$nB_{X^*}$ and so a $G_\delta$ subset of this space. 
Since $X^*\setminus A=\bigcup_{n=1}^\infty(nB_{X^*}\setminus A_n)$ and each of the sets 
$nB_{X^*}\setminus A_n$ is a $\sigma$--compact subset of $nB_{X^*}$ we get that $A$ is a 
$G_\delta$ set in $X^*$.
\end{proof}
\begin{thm}
\label{T:Polish}
Every Polish Banach space is strongly $\K$.
\end{thm}
\begin{proof}
Let $X$ be a Polish Banach space. The closed unit ball $B_X$ of $X$ is a $G_\delta$ subset of the closed 
unit ball $(B_{X^{**}}, w^*)$ of $X^{**}$. Moreover 
each ball $nB_X$ is a $G_\delta$ subset of $nB_{X^{**}}$ for $n\geq1$. 
So, as in the proof of assertion (iii) of Lemma \ref{L:w*} (and since $X^*$  is separable) 
we get that $X$ is a $G_\delta$ subset of $(X^{**}, w^*)$. The conclusion follows from assertion 
(ii) of the same Lemma.
\end{proof}
\begin{rem}
\textbf{(a)} From Cor. \ref{C:reflex} if a Polish Banach space is non reflexive 
(for example, the predual of JT space of James), then it does not have property (*), 
especially it is not SWCG with respect to any superspace of it.

\textbf{(b)} A weak$^*$--Polish Banach space (according to Rosenthal's definition (\cite{R2}) is 
a separable Banach space $X$, that embeds to the dual of a (separable) Banach space $Y$ 
such that $B_X$ is a Polish space in the weak$^*$ topology of $Y^*$. It is clear that 
a Polish Banach space $X$ is weak$^*$--Polish in $X^{**}$. This wider class of 
Banach spaces is studied by Rosenthal in \cite{R2}. We note that the argument of the previous theorem 
with Lemma \ref{L:w*} imply that every weak$^*$--Polish Banach space is a strongly $\K$ 
subset of $(Y^*, w^*)$.

\textbf{(c)} The result stated in Th. \ref{T:Polish} is also true and for Banach spaces with 
$(B_X, w)$ \v{C}ech complete space, because every such space is isomorphic to a direct sum 
of a Polish and a reflexive Banach space (cf. \cite{EW} and Prop. \ref{P:fp} below).
\end{rem}
\begin{prop}\label{P:fp}
The class of strongly $\K$ Banach spaces is closed under closed subspaces and finite products.
\end{prop}
\begin{proof}
It follows immediately from Rem. \ref{R:Kcd} (ii) and from Prop. \ref{finiteprod}.
\end{proof}

It is also proved the analogue of \cite[Prop. 5.1]{KM}.
(cf. also Th. \ref{T:sum}).
\begin{thm}\label{T:prodKcd}
Let $(X_n)$ be a sequence of (strongly) $\K$ Banach spaces and $1\leq p< \infty$. 
Then the space $X=\left( \sum_{n=1}^\infty \oplus X_n\right)_p$
is (strongly) $\K$.
\end{thm}
\begin{proof}
For every $n \in \N $ put $Y_n= \left(\sum_{k=1}^n\oplus X_k \right)_p$ 
and	$Z_n=\left(\sum_{k=n+1}^\infty \oplus X_k\right)_p$. 
So	
$X=Y_n \oplus Z_n$ and $X^{**} =Y_n^{**} \oplus Z_n^{**}$, 
where 	$Y_n^{**}=\left(\sum_{k=1}^n \oplus X_k^{**} \right)_p$. 
More precisely $X=(Y_n\oplus Z_n)_p$ and $X^{**}=(Y_n^{**}\oplus Z_n^{**})_p$.

Let $p>1$. 
It follows from Prop. \ref{finiteprod} that $Y_n \oplus Z_n^{**}$ is a (strongly) $\K$ subset of 
$(X^{**}, w^*)$, thus $Y_n +\dfrac{1}{n}B_{Z_n^{**}}$ is also a (strongly) $\K$ 
subset of $(X^{**}, w^*)$ for every $n \in \N $. Let 
$x^{**} \in \bigcap _{n=1}^\infty (Y_n + \frac{1}{n}B_{Z_n^{**}})$. 
Then there exist 
$y_n \in Y_n$ and $z_n^{**} \in B_{Z_n^{**}}$ such that 
$ x^{**}=y_n+\frac{1}{n}z_n^{**} $ for all $n \in \N$, 
hence $\| x^{**}-y_n \| \leq \frac{1}{n}$ for all $n \in \N$. 
It follows that $y_n \stackrel{\|\cdot \|}{\rightarrow}x^{**}$. Moreover 
$\{y_n :n \in \N \} \subseteq \bigcup _{n=1}^\infty Y_n$ and the subspace 
$\bigcup _{n=1}^\infty Y_n$ is dense in $X=\left( \sum_{n=1}^\infty \oplus X_n\right)_p$, so $x^{**} \in X$.
Thus $\bigcap_{n=1}^\infty\left(Y_n+\frac{1}{n}B_{Z_n^{**}} \right)=X$ 
and from Prop. \ref{P:intersection} $X$ is a (strongly) $\K$ subset of $(X^{**}, w^*)$.

Let $p=1$. 
Put $E=\left(\sum_{k=1}^\infty \oplus X_k^{**} \right)_1$; clearly $E$ is a subspace of $X^{**}$.
We claim that $E$ is a strongly $\K$ subset of $(X^{**}, w^*)$. We first prove that 
\begin{equation} \label{eq:inter}
E=\bigcap_{p=1}^\infty \bigcup_{n=1}^\infty \left(Y_n^{**}+\frac{1}{p}B_{Z_n^{**}} \right) \, .
\end{equation}
Since each $Y_n^{**}+\dfrac{1}{p}B_{Z_n^{**}}$ is a $\sigma$--compact subset of $(X^{**}, w^*)$, 
if equality (\ref{eq:inter}) holds true, then  
$E$ is a $\K$ subset of $(X^{**}, w^*)$. Let 
$x^{**}=(x_n^{**})_n \in E$. 
Then $\sum_{n=1}^\infty \| x_n^{**}\| <+\infty$. So for every $p \in \N$ there exists 
$n(p) \in \N $ such that $\sum_{k=n(p)+1}^\infty \|x_k^{**}\| <\frac{1}{p}$.
It then follows that given $p \in \N$ we have that,  
	\[x^{**}=(x_1^{**}, x_2^{**}, \dots x_{n(p)}^{**},0,0 \dots)+
	(0,0, \dots 0,x_{n(p)+1}^{**}, \ldots ) 
	\in Y_{n(p)}^{**}+\frac{1}{p}B_{Z_{n(p)}^{**}} \, ,
\]
hence 
$x^{**}\in \bigcap_{p=1}^\infty \bigcup_{n=1}^\infty \left(Y_n^{**}+\frac{1}{p}B_{Z_n^{**}} \right)$.

Assume that
$x^{**} \in \bigcap_{p=1}^\infty \bigcup_{n=1}^\infty \left(Y_n^{**}+
	\frac{1}{p}B_{Z_n^{**}} \right)$.
Then for each $p \in \N$ there exists $n(p) \in \N$ such that 
$x^{**} \in Y_{n(p)}^{**}+\frac{1}{p}B_{Z_{n(p)}^{**}}$,
so for every $p \in \N$ there exist $y_p^{**} \in Y_{n(p)}^{**}$ and 
$z_p^{**} \in B_{Z_{n(p)}^{**}}$, such that $x^{**}=y_p^{**}+\frac{1}{p}z_p^{**}$. 
It follows that $\| x^{**}-y_p^{**}\| \leq \dfrac{1}{p}$  for all 
$p \in \N$, which implies that 
$y_p^{**} \to x^{**} \in X^{**}$. As for every $p \in \N$ we have 
$y_p^{**} \in Y_{n(p)}^{**} \subseteq E=\left(\sum_{k=1}^\infty \oplus X_k^{**} \right)_1
	\subseteq X^{**}\,$ ,
it follows that $x^{**} \in E$ and equality (\ref{eq:inter}) has been proved. 
Now let $\Omega \subseteq E$ be a $w^*$--compact subset in $X^{**}$. 
Then using a ''gliding hump'' argument we can prove that for every $\varepsilon >0$ 
there exists $N \in \N$ such that $\sum_{n=N}^\infty \|x_n^{**} \| <\varepsilon $ 
for all $x^{**}=(x_n^{**}) \in \Omega $ (cf. \cite[Prop. 1.2]{SW}). 
This fact together with equality (\ref{eq:inter}), easily imply 
that $E$ is a strongly $\K$ subset of $(X^{**}, w^*)$.
It is now clear that for every $n \in \N$ the space 
$\left( \sum _{k=n+1}^\infty \oplus X_k^{**}\right)_1$ 
is a strongly $\K$ subset of $Z_n^{**}$. As the space 
$Y_n= \left(\sum_{k=1}^n\oplus X_k \right)_1$
is a (strongly) $\K$ subset of $(Y_n^{**}, w^*)$, we conclude that for every 
$n \in \N$ the space 
$\left[\left(\sum_{k=1}^n \oplus X_k \right)_1 
	\oplus \left(\sum_{k=n+1}^\infty X_k^{**} \right)_1\right]_1 $
is a (strongly) $\K$ subset of $\left(Y_n^{**} \oplus Z_n^{**}\right)_1=X^{**}$\, . 
It then follows that the space  
\[\bigcap_{n=1}^\infty \left[ Y_n \oplus \left( \sum _{k=n+1}^\infty X_k^{**}\right)_1\right]_1=X
\]
is a (strongly) $\K$ subset of $X^{**}$.
\end{proof}
Subsequently we examine locally (in the weak topology of a Banach space) 
the property of strong 
$\mathcal{K}$--analyticity and the property of a set to be 
strongly $\K$.
\begin{prop}
Let $X$ be a strongly $\K$ Banach space. Then every ball (open or closed) of $X$ is a strongly $\K$ 
subset of $(X^{**}, w^*)$.
\end{prop}
\begin{proof}
Clearly, every closed ball of $X$ is a strongly $\K$ subset of $(X^{**}, w^*)$, 
as a closed subset of $(X, w)$.

We are going to prove that the open unit ball $B_X^o$ of $X$ is a strongly $\K$ subset of $X^{**}$.
It holds that
$	B_X^o=X\bigcap B_{X^{**}}^o $ and 
$B_{X^{**}}^o= \bigcup _{m=1}^\infty (1-\frac{1}{m})B_{X^{**}}$.
The space $(B_{X^{**}}^o, w^*)$ is hemicompact, because for each $w^*$--compact subset $L$ 
of $B_{X^{**}}^o$ there exists $m \in \mathbb{N}$ such that
$L \subseteq (1-\frac{1}{m})B_{X^{**}}$.
Consequently $(B_{X^{**}}^o, w^*)$ is a strongly $\K$ subset of $X^{**}$, so the open 
unit ball of $X$ is a strongly $\K$ subset of $X^{**}$, as intersection of two strongly 
$\K$ subsets.

In a similar way we can prove that any open ball of $X$ is a strongly $\K$ 
subset of $(X^{**}, w^*)$.
\end{proof}
\begin{rem}\label{R:local}
In the same manner we can prove that if $X$ is $\K$, then every ball of $X$ 
is a $\K$ subset of $(X^{**}, w^*)$. If in addition $X$ is separable, then 
it is easy to see that every norm $G_\delta$ subset of $X$ is a $\K$ subset 
of $(X^{**},w^*)$.
\end{rem}
\begin{prop} \label{P:sequence}
Let $X$ be a Banach space and $A \subseteq X$. Then the following are equivalent:
\begin{enumerate}[\upshape(i)]
	\item $(A, w)$ is a strongly $\K$ subset of $(X^{**}, w^*)$.
	\item There exists a family $\{K_{n,m} :n, m \in \mathbb{N} \}$ of $w^*$--compact subsets 
	of $X^{**}$ such that:
\begin{enumerate}
	\item [(a)] $A=\bigcap_{n=1}^\infty \bigcup _{m=1}^\infty K_{n,m}$
	\item [(b)] For each sequence $(x_k)$ of $A$, which converges weakly 
	to some $x$ in $A$ and for each $n \in \mathbb{N}$ there exists $m \in \mathbb{N}$ 
	such that $\{x_k : k \in \mathbb{N} \} \subseteq K_{n,m}$.
\end{enumerate}
\end{enumerate}
\end{prop}
\begin{proof}
The proof uses the sequential compactness of a weakly compact set and as it is simple we omit it. 
\end{proof}
\begin{lem} 
\label{L:Schur}
Let $X$ be a separable Banach space and $V$ an open subset of $X$. Then there exists 
a countable family $\{ V_n : n \in \mathbb{N} \}$ of closed balls of $X$ 
such that $V= \bigcup _{n=1}^\infty V_n$ with the following property: 
For every sequence $(x_n)$ of $X$, $x \in X$ with $x_n \stackrel{\| \cdot \|}{\rightarrow}x$ 
and $L=\{x_n: n\in \mathbb{N} \} \bigcup \{x \} \subseteq V$ there exists 
$N \in \mathbb{N}$ such that $L \subseteq \bigcup_{n=1}^N V_n$.
\end{lem}
\begin{proof}
The proof is simple, so we omit it.
\end{proof}
\begin{lem}
\label{L:Polish}
Let $X$ be a Polish space and $\{F_\sigma : \sigma \in \varSigma \}$ be a family 
of closed subsets of $X$ such that:
\begin{enumerate}[\upshape (i)]
	\item If $\sigma \leq \tau$, then $F_\sigma \subseteq F\tau$.
	\item For every compact subset $K$ of the space $Y=\bigcup _{\sigma \in \varSigma} F_\sigma $ 
	there exists	$\sigma \in \varSigma$ such that $K \subseteq F_\sigma$.
\end{enumerate}
Then $Y$ is Polish.
\end{lem}
\begin{proof}
As $X$ is a Polish space there exists a family 
$\{X_\sigma : \sigma \in \varSigma \}$ of compact subsets of $X$ such that:
\begin{enumerate}
	\item [(a)] If $\sigma \leq \tau$, then $X_\sigma \subseteq X\tau$
	\item [(b)] For each compact subset $K$ of $X$ there exists 
	$\sigma \in \varSigma$ such that $K \subseteq X_\sigma$
\end{enumerate}
(see \cite[Th. 3.3]{Chr} or \cite[Th. 1.8]{MS}). 
Put
$K_\tau^\sigma=F_\sigma \bigcap X_\tau $ 
for all $\sigma ,\tau \in \varSigma$.
Clearly every $ K_\tau^\sigma $ is compact and also for every $\sigma \in \varSigma$ we have 
$F_\sigma=\bigcup_{\tau \in \varSigma}(F_\sigma \bigcap X_\tau)=
	\bigcup_{\tau \in \varSigma} K_\tau ^\sigma $,
so 
$Y=\bigcup _{\sigma \in \varSigma} F_\sigma =\bigcup _{\sigma, \tau \in \varSigma} K_\tau ^\sigma $.
We remark that:
\begin{enumerate}
	\item [(a)] If $\tau _1 \leq \tau _2$ and $\sigma _1 \leq \sigma _2$, then 
	$K_{\tau _1}^{\sigma _1}=F_{\sigma_1}\bigcap X_{\tau_1} \subseteq 
	F_{\sigma_2}\bigcap X_{\tau_2}= K_{\tau _2}^{\sigma _2}$
	\item [(b)] If $K$ is a compact subset of $Y$, then there exist 
	$\sigma$, $\tau \in \varSigma$ such that $K \subseteq K_\tau^\sigma$. 
	(From (ii) and (b) there exists $\sigma \in \varSigma$ such that 
	$K \subseteq F_\sigma$ and $\tau \in \varSigma$ such that 
	$K \subseteq F_\sigma \bigcap X_\tau =K_\tau^\sigma$.)
\end{enumerate}
Then by \cite[Th. 3.3]{Chr} the conclusion follows.
\end{proof}

\begin{prop} \label{P:Gdelta}
Let $X$ be a separable Banach space and $A \subseteq X$. If $(A, w)$ is strongly 
$\mathcal{K}$--analytic, then $(A, \| \cdot \|)$ is strongly $\mathcal{K}$--analytic, 
equivalently $A$ is a norm $G_\delta$ subset of $X$.
\end{prop}
\begin{proof}
As $(A, w)$ is strongly $\mathcal{K}$--analytic, there exist compact sets 
$\{A_\sigma : \sigma \in \varSigma \}$ such that:
\begin{enumerate}[\upshape (i)]
	\item If $\sigma \leq \tau$, then $A_\sigma \subseteq A_\tau$.
	\item For every weakly compact subset $K$ of $A$ there exists 
	$\sigma \in \varSigma $, such that $K \subseteq A_\sigma$.
\end{enumerate}
Then the sets $A_\sigma$ are weakly compact, so norm--closed in $X$. As every norm compact 
set is weakly compact, by Lemma \ref{L:Polish} the space 
$(A, \| \cdot \|)$ is Polish, so $A$ is a norm $G_\delta$ subset of $X$.
\end{proof}
The following results assure that in certain classes of separable Banach spaces $X$, 
a set $A\subseteq X$ is strongly $\mathcal{K}$--analytic in $(X,w)$ 
if and only if $A$ is a strongly $\K$ set in $(X^{**},w^*)$.
\begin{thm}
\label{P:KG}
Let $X$ be a separable Banach space with Schur property and $A \subseteq X$. 
Then the following are equivalent:
\begin{enumerate}[\upshape (i)]
	\item $(A, w)$ is strongly $\mathcal{K}$--analytic
	\item $(A, \| \cdot \|)$ is a $G_\delta$ υsubset of $(X, \| \cdot \|)$ (equivalently a Polish space)
	\item $A$ is a strongly $\K$ subset of $(X^{**}, w^*)$.
\end{enumerate}
\end{thm}
\begin{proof}  \textbf{(i) $\Rightarrow$ (ii)} It is immediate from Prop. \ref{P:Gdelta}.

\textbf{(ii)$ \Rightarrow $ (iii)} It is enough to prove that the result holds for norm open subsets of $X$. 
Let $V$ be a norm open subset of $X$. As $X$ is separable $V$ is written as a countable union 
of closed balls, as in Lemma \ref{L:Schur}. Let 
$	V=\bigcup _{n=1}^\infty B[y_n, r_n]=\bigcup_{n=1}^\infty V_n$. 
Then	$V=X\bigcap \left(\bigcup _{n=1}^\infty B_{X^{**}}[y_n, r_n] \right)$.
The space $X$ is separable Banach space with Schur property, so it is SWCG. Let $K$ be a weakly 
compact subset of $X$, that strongly generates $X$. Then 
	\[X=\bigcap_{n=1}^\infty \bigcup_{m=1}^\infty (mK+\frac{1}{n}B_{X^{**}})
\]
and 
	\[V=\bigcap_{n=1}^\infty \left[\bigcup_{m=1}^\infty (mK+\frac{1}{n}B_{X^{**}})\right]
	\cap \left[\bigcup_{n=1}^\infty \left(\bigcup _{m=1}^n B_{X^{**}}[y_m, r_m]\right)\right]\, .
\]
Assume $x_n \stackrel{w}{\rightarrow}x$ and
$L=\{x_n: n\in \mathbb{N} \} \bigcup \{x \} \subseteq V$. The set $L$ is weakly compact, 
so for each $n \in \mathbb{N}$ there exists $m \in \mathbb{N}$ such that 
$L \subseteq mK+\frac{1}{n}B_{X^{**}}$. 
Furtheremore $X$ has Schur property, so $x_n \stackrel{\| \cdot \|}{\rightarrow}x$. 
Then according to Lemma \ref{L:Schur} there exists $m \in \mathbb{N}$ such that 
$L \subseteq \bigcup_{n=1}^m V_n $, so 
$L \subseteq  \bigcup_{n=1}^m B_{X^{**}}[y_n, r_n] $, 
which completes the proof.

\textbf{(iii) $\Rightarrow $ (i)} Every strongly $\K$ subset of 
$(X^{**}, w^*)$ is strongly $\mathcal{K}$--analytic space.
\end{proof}
The following result is an immediate consequence of Lemma \ref{L:w*}.
\begin{prop}
Let $X$ be a Banach with separable dual and $A\subseteq X$. Then the following are equivalent:
\begin{enumerate}[\upshape (i)]
	\item $(A, w)$ is strongly $\mathcal{K}$--analytic.
	\item $(A, w)$ is strongly $\K$ subset of $(X^{**}, w^*)$.
	\item $(A, w)$ is a $G_\delta$ subset of $(X^{**}, w^*)$.
\end{enumerate}
\end{prop}
\begin{cor}
Let $X$ be a Banach space with separable dual. The following are equivalent:
\begin{enumerate}[\upshape (i)]
	\item $X$ is SWKA.
	\item $B_X$ is strongly $\K$ subset of $(B_{X^{**}}, w^*)$.
	\item $X$ is Polish.
	\item $X$ is strongly $\K$. 
\end{enumerate}
\end{cor}
\begin{proof} 
 It is immediate from the previous proposition and Th. \ref{T:Polish}. 
\end{proof}

We now mention some open questions.

\textbf{(A)} Let $X$ be a Banach space. 
\begin{enumerate}[\upshape (i)]
	\item If $X$ is SWKA, is then strongly $\K$ or at least $\K$?
	\item If $X$ is SWKA with unconditional basis, is then $X$ strongly $\K$?
\end{enumerate}
We note that there exists a WKA Banach space, that is not $\K$ and in addition every 
WKA Banach space with unconditional basis is $\K$ (\cite{AAM1}, \cite{AAM2}).

\textbf{(B)} Let $X$ be a Banach space. 
\begin{enumerate} [\upshape (i)]
	\item Assume that the closed unit ball $B_X$ of $X$ is a $\K$ subset of $(X^{**}, w^*)$. 
	Is then $X$ itself a $\K$ subset of $(X^{**}, w^*)$? (If $X$ is separable, then it is of course 
	a $\K$ subset of $(X^{**}, w^*)$.)
	\item Assume that the closed unit ball $B_X$ of $X$ is a strongly $\K$ subset of $(X^{**}, w^*)$. 
	Is then $X$ itself a strongly $\K$ subset of $(X^{**}, w^*)$?
\end{enumerate}

\textbf{(C)} Let $X$ be a separable Banach space. 
\begin{enumerate}[\upshape (i)]
	\item Assume that $X$ is strongly $\K$. Does there exists a norm $G_\delta$ subset of $X$, 
	that is not a strongly $\K$ subset of $(X^{**}, w^*)$? (cf. Remark \ref{R:local}.)
	\item Let $X=C[0,1]$. Does there exists a (necessarily norm $G_\delta$) strongly $\mathcal{K}$--analytic 
	subset of $(X, w)$, that is not a strongly $\K$ subset of $(X^{**}, w^*)$? (cf. Question \ref{Q:A}.)
\end{enumerate}

It is worth noting that questions (A) and (B)(ii) make sense even in the class of separable 
Banach spaces.

\end{document}